\documentclass[12pt]{amsart}

\usepackage{amsthm,amsmath,amssymb}
\usepackage{graphicx}
\usepackage[colorlinks=true,citecolor=black,linkcolor=black,urlcolor=blue]{hyperref}
\usepackage{mathrsfs}
\usepackage{fancyhdr}
\theoremstyle{plain}
\newtheorem{theorem}{Theorem}
\newtheorem{lemma}[theorem]{Lemma}
\newtheorem{corollary}[theorem]{Corollary}

\theoremstyle{definition}
\newtheorem{definition}[theorem]{Definition}
\newtheorem{example}[theorem]{Example}

\theoremstyle{remark}
\newtheorem{remark}[theorem]{Remark}

\begin{document}
\title{Characterization of regular checkerboard colourable twisted duals of ribbon graphs}
\author{Xia Guo}
\address{School of Mathematical Sciences, Xiamen University, 361005, Xiamen, China}
\email{guoxia@stu.xmu.edu.cn}
\author{Xian'an Jin}
\address{School of Mathematical Sciences, Xiamen University, 361005, Xiamen, China}
\email{xajin@xmu.edu.cn}
\author{Qi Yan}
\address{School of Mathematical Sciences, Xiamen University, 361005, Xiamen, China}
\email{qiyanmath@163.com }
\thanks{This work is supported by NSFC (No. 11671336) and the Fundamental Research
Funds for the Central Universities (No. 20720190062). }

\begin{abstract}
The geometric dual of a cellularly embedded graph is a fundamental concept in graph theory and also appears in many other branches of mathematics. The partial dual is an essential generalization which can be obtained by forming the geometric dual with respect to only a subset of edges of a cellularly embedded graph. The twisted dual is a further generalization by combining the partial Petrial. Given a ribbon graph $G$, in this paper, we first characterize regular partial duals of the ribbon graph $G$ by using spanning quasi-tree and its related shorter marking arrow sequence set. Then we characterize checkerboard colourable partial Petrials for any Eulerian ribbon graph by using spanning trees and a related notion of adjoint set. Finally we give a complete characterization of all regular checkerboard colourable twisted duals of a ribbon graph, which solve a problem raised by Ellis-Monaghan and Moffatt [\emph{T. Am. Math. Soc.}, \textbf{364(3)} (2012), 1529-1569].
\end{abstract}



\maketitle
\section{Introduction}
\noindent

A cellularly embedded graph, a ribbon graph \cite{BR02}, a band decomposition \cite{GT}, an arrow presentation \cite{C09} and a signed rotation system \cite{GT,Mohar} arising in many different contexts can be used to describe the same object. We sometime use the term ``embedded graph" loosely to mean any of the above representations of graphs on surfaces and refer the reader to the monograph \cite{EM13} for the equivalence among them.

The geometric dual of a cellularly embedded graph is a fundamental concept in graph theory and also appears in many other branches of mathematics.
To unify several Thistlethwaite's theorems \cite{Chmu,Chmu2,Lin1} on relations between the Jones polynomial of virtual links and the topological Tutte polynomial of ribbon graphs constructed from virtual links in different ways,
S. Chmutov, in \cite{C09}, introduced the concept of \emph{partial dual} of a cellularly embedded graph in terms of ribbon graphs. Roughly
speaking, a partial dual is obtained by forming the geometric dual with respect to only a subset of edges of a cellularly embedded graph. Let $G$ be an embedded graph and $A\subseteq E(G)$. We denote by $G^{\delta(A)}$ the partial dual of $G$ with respect to $A$.

It is well known that a plane graph is Eulerian if and only if its geometric dual is bipartite. In \cite{HM}, Huggett and Moffatt extended this result to partial duals of plane graphs and characterized all bipartite partial duals of a plane graphs in terms of all-crossing directions in its medial graph. In \cite{MJ}, the authors characterized all Eulerian partial duals using semi-crossing directions in its medial graph. In \cite{DJM}, the authors characterized all bipartite and Eulerian partial duals for any ribbon graph.

The Petrie dual of a cellularly embedded graph $G$ is formed by retaining the vertices and edges of $G$, but
removing its faces, and replacing them with new faces bounded by the Petrie polygons of $G$  \cite{Jo,WI}. Thus the
underlying graph of the   embedded graph is unchanged, but the surface may be totally different. In the language of a ribbon graph, it is equivalent to giving each edge a half-twist. Let $G$ be a ribbon graph and $A\subseteq E(G)$. The \emph{partial Petrial},
$G^{\tau(A)}$, of $G$ with respect to $A$ is the ribbon graph obtained from $G$ by giving a half-twist to each of the edges in $A$.

In \cite{EM12}, Ellis-Monaghan and Moffatt introduced the concept of twisted duality by combining partial dual and partial Petrial together. It has a number of applications in graph theory, knot theory, matroid theory and so on
\cite{CN1,CN2,QJK,EM15,KRV11,M10,M11,V09,V11}. The twisted dual is also closely related to the medial graph. Let $G$ be a cellularly embedded graph and $G_m$ its medial graph. Actually it is shown in \cite{EM12} that all twisted duals of a cellularly embedded graph $G$ is precisely the set of cellularly embedded graphs whose medial graphs isomorphic (as abstract graphs) to $G_m$. It is clear that a medial graph is $4$-regular checkerboard colourable.
Therefore, in \cite{EM12},  Ellis-Monaghan and Moffatt posed the following two problems:
\begin{enumerate}
\item If $G$ is a 4-regular embedded graph, which of its twisted duals are also 4-regular and checkerboard colorable?
\item Is it possible to characterize those embedded graphs, without degree restrictions, that have a checkerboard colourable twisted dual?
\end{enumerate}

The rest of this paper is organized as follows. In Section 2 we give precise definitions of ribbon graphs, arrow presentations and twisted duality, and also some basic properties of twisted duality.
In Section 3 we introduce the notion of shorter marking arrow sequence set and provide a necessary and sufficient condition
for $G^{\delta(A)}$ to be regular for any ribbon graph $G$ and $A\subseteq E(G)$ in terms of its spanning quasi-tree and the related shorter marking arrow sequence set.
Then we characterize all checkerboard colourable partial Petrials of an Eulerian ribbon graph in terms of its spanning tree and the related notion of adjoint set in Section 4.
In the final Section 5, we give a complete characterization of all regular checkerboard colourable twisted duals of a ribbon graph and an example is provided to explain this characterization.

\section{Preliminaries}
\noindent

A \emph{cellularly embedded graph} is a graph $G$ embedded in a closed surface $\Sigma$ such that every connected component of $\Sigma-G$ is a 2-cell, called a \emph{face} of the cellularly embedded graph. Two cellularly embedded graphs $G\subset \Sigma$ and $G'\subset \Sigma'$ are equivalent, written $G=G'$, if there is a homeomorphism from $\Sigma$ to $\Sigma'$ that sends $G$ to $G'$. As it is much more convenient for our purpose, we shall realise cellularly embedded graphs as ribbon graphs or arrow presentations. In this section we provide a brief overview of ribbon graphs and arrow presentations and  refer the reader to \cite{EM13} for details.

\begin{definition}\cite{BR02}
A ribbon graph $G$ is a (possibly non-orientable) surface with boundary represented as the union of two sets of discs, a set $V(G)$ of vertices and a set $E(G)$ of edges such that\\
(1) the vertices and edges intersect in disjoint line segments, we call them {\it common line segments} as in \cite{Me11};\\
(2) each such common line segment lies on the boundary of precisely one vertex and precisely one edge;\\
(3) every edge contains exactly two such common line segments.

\end{definition}

Let $G=(V(G), E(G))$ be a ribbon graph. By deleting the common line segments from the boundary of a vertex $v\in V(G)$,
we obtain disjoint line segments, called \emph{vertex line segments} of $v$ \cite{Me11}.
By deleting common line segments from the boundary of an edge $e\in E(G)$, we obtain two disjoint line segments,
called \emph{edge line segments} of $e$ \cite{Me11}. See Figure \ref{f 7}.
A ribbon graph $H$ is a ribbon subgraph of $G$ if $V(H)\subseteq V(G)$ and $E(H)\subseteq E(G)$. Furthermore if $V(H)=V(G)$, then $H$ is said to be a spanning ribbon subgraph of $G$. Note that every subset $A$ of $E(G)$ uniquely determines a spanning ribbon subgraph $G[A]=(V(G),A)$ of $G$.

\begin{figure}[htbp]
\centering
\includegraphics[width=3.5in]{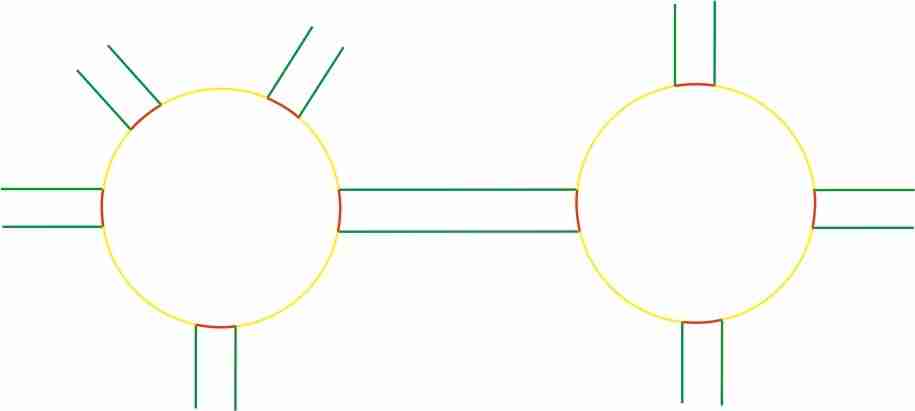}
\caption{Common line segments (red), vertex line segments (yellow) and edge line segments (green).}
\label{f 7}
\end{figure}

\begin{definition}\cite{C09}
An arrow presentation consists of a set of circles (corresponding to vertices) with pairs of labelled arrows (corresponding to edges), called  marking arrows, on them such that there are exactly two marking arrows of each label.
\end{definition}

Both ribbon graphs and arrow presentations are equivalent to cellularly embedded graphs. Two ribbon graphs or arrow presentations are equivalent if they are equivalent as cellularly embedded graphs.
We emphasise that the circles in an arrow presentation are not equipped with any embedding in the plane or $R^3$.

Let $C_{1}$ and $C_{2}$ be two cycles in a ribbon graph $G$. Suppose $v$ is a vertex in both cycles, with edges $e_1,e_2\in C_1$ and edges $f_1,f_2\in C_2$ all incident with $v$ (possibly $e_1=e_2$ or $f_1=f_2$ if $C_{1}$ or $C_{2}$ is a loop).
We say that $C_1$ and $C_2$ alternate at $v$ if these edges $e_1,e_2,f_1,f_2$ appear in the cyclic order as $(e_1\cdots f_1\cdots e_2\cdots f_2 \cdots)$ along the boundary of $v$. A loop $C$ at a vertex $v$ of a ribbon graph $G$ is trivial if there is no cycle or loop in $G$ which alternates with $C$.

Let $G$ be an arrow presentation and $A\subseteq E(G)$. The {\it partial Petrial} of $G$ with respect to $A$ is obtained by reversing the direction of exactly one of the labelled arrows of each edge of $A$, denoted by $G^{\tau(A)}$. The partial Petrial of a ribbon graph $G$ with respect to an edge $e\in E(G)$ is shown in Figure \ref{f 8} (a).

\begin{figure}[htbp]
\centering
\includegraphics[width=5in]{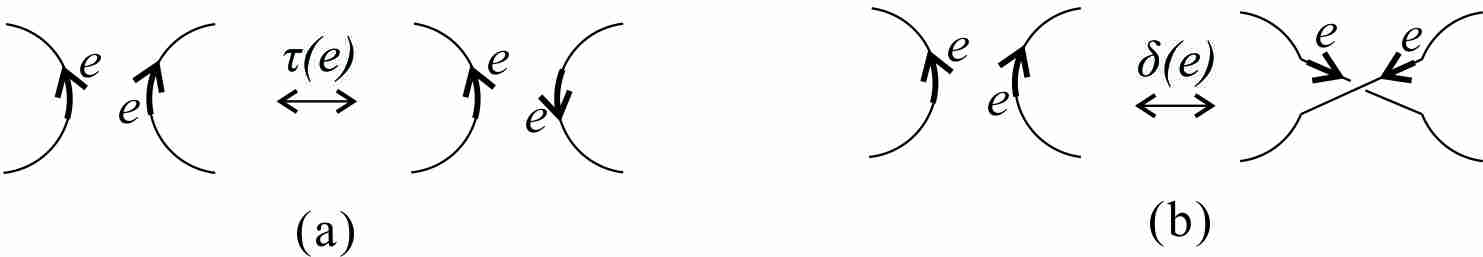}
\caption{Taking the partial Petrial (a) and partial dual (b) of an edge in an arrow presentation.}
\label{f 8}
\end{figure}

The {\it partial dual} $G^{\delta(A)}$ of a ribbon graph $G$ with respect to $A$ is constructed as follows. For each $e\in A$, suppose $e',~ e''$ are the two arrows labelled $e$ in the arrow presentation of $G$. Draw a line segment with an arrow on it directed from the head of $e'$ to the tail of $e''$ and from the head of $e''$ to the tail of $e'$, respectively. Label both of these arrows $e$, then delete $e',~e''$ and the arcs containing them, as shown in Figure \ref{f 8} (b).

Let $A, B\subseteq E(G)$. Then $A\backslash B:=A-B$. We denote by $G/A$ the ribbon graph obtained from $G$ by contracting each edge of $A\subseteq E(G)$ and then $G/A:=G^{\delta(A)}-A$.
Throughout this paper, we will often omit the set brackets in the case of a single element set, i.e. we write $G/e$ instead of $G/\{e\}$.

\begin{lemma}\cite{EM13}\label{l 4}
Let $G$ be a ribbon graph and $A\subseteq E(G)$. Then $G$ is orientable if and only if $G^{\delta(A)}$ is orientable.
\end{lemma}

\begin{definition}\cite{EM12}
Two ribbon graphs $G$ and $H$ are twisted duals if there exist $A_1$,$\cdots$, $A_6\subseteq E(G)$ such that
$$H=G^{1(A_1)\delta(A_2)\tau(A_3)\tau\delta(A_4)\delta\tau(A_5)\tau\delta\tau(A_6)},$$
where the $A_i$ partition $E(G)$.
\end{definition}

We have the following lemma.

\begin{lemma}\label{l 2}
Two ribbon graphs $G$ and $H$ are twisted duals if and only if there exist $B_1, B_2, B_3\subseteq E(G)$ such that\\
$$H=G^{{\tau(B_1)}{\delta(B_2)}{\tau(B_3)}}.$$
\end{lemma}

\begin{proof} Let $A\subseteq E(G)$. Denote by $\bar{A}$ the complement of the subset $A$, i.e. $\bar{A}=E(G)-A$. We take $A_1=(\bar{B_1}\cap \bar{B_2}\cap \bar{B_3})\cup(B_1\cap \bar{B_2}\cap B_3)$, $A_2=\bar{B_1}\cap B_2\cap \bar{B_3}$, $A_3=(B_1\cap \bar{B_2}\cap \bar{B_3})\cup(\bar{B_1}\cap \bar{B_2}\cap B_3)$, $A_4=B_1\cap B_2\cap \bar{B_3}$, $A_5=\bar{B_1}\cap B_2\cap B_3$, $A_6=B_1\cap B_2\cap B_3$, then $G^{1(A_1)\delta(A_2)\tau(A_3)\tau\delta(A_4)\delta\tau(A_5)\tau\delta\tau(A_6)}=G^{{\tau(B_1)}{\delta(B_2)}{\tau(B_3)}}$. Conversely, we take $B_1=A_3\cup A_4\cup A_6$, $B_2=A_2\cup A_4\cup A_5\cup A_6$, $B_3=A_5\cup A_6$, then $G^{{\tau(B_1)}{\delta(B_2)}{\tau(B_3)}}=G^{1(A_1)\delta(A_2)\tau(A_3)\tau\delta(A_4)\delta\tau(A_5)\tau\delta\tau(A_6)}$.
\end{proof}

Let $G$ be a ribbon graph. Then there is a group action of $\mathfrak{G}=~<\delta,\tau|{ \delta}^2, {\tau}^2, {(\delta\tau)}^3>$ on $(G,E_1)$ given by $g(G,E_1):=(G^{g(E_1)}, E_1)$ for $g\in \mathfrak{G}$, $E_1\subseteq E(G)$.

\begin{lemma}\cite{EM13}\label{l 1}
Let $G$ be a ribbon graph, $A$,$B \subseteq E(G)$ and $g,h \in \mathfrak{G}$. Then the following hold:\\
(1) There is a natural bijection between the edges of $G$ and the edges of $G^{g(A)}$.\\
(2) If $A\cap B = \emptyset$, then $G^{g(A)h(B)}=G^{h(B)g(A)}$.\\
(3) $G^{g(A)}=(G^{g(e)})^{g(A \backslash \{ e \} )}$, when $e\in A$, and thus twisted duals can be formed one edge at a time.\\
\end{lemma}

In this paper, we shall identify the edges of $G$ with those of $G^{g(A)}$ under the natural bijection.
A ribbon graph is {\it checkerboard colourable} if there is an assignment of the two colours, say, black and white, to its boundary components such that two edge line segments of any edge receive different colours.
If a ribbon graph $G$ has equipped with a checkerboard colouring, then the adjacent vertex line segments of each vertex of $G$ must have different colours, which implies that each vertex of a checkerboard colourable ribbon graph has even number of vertex line segments, i.e. a checkerboard colourable ribbon graph is Eulerian (here we do not require that an Eulerian graph must be connected).

A {\it bouquet} is a ribbon graph with exactly one vertex. A {\it quasi-tree} is a ribbon graph with exactly one boundary component. Notice that there is a natural bijection between the vertex boundaries of $G^{\delta(A)}$ and the boundary components of $G[A]$. As a consequence, we have the following.

\begin{lemma}\label{l 6}
Let $G$ be a ribbon graph and $A\subseteq E(G)$. Then $G^{\delta(A)}$ is a bouquet if and only if $G[A]$ is a quasi-tree of $G$.
\end{lemma}

In the following we will always assume that the ribbon graph is connected and will do so without further comment.

\section{The regular partial duals}
\noindent

In this section we study the regular partial duals of a ribbon. We first give the definition of shorter marking arrow sequence set.

Let $C$ be the (unique) circle of an arrow presentation of a bouquet $G$ with $m$ edges $e_1,\cdots, e_m$, $m\in \mathbb{N}$. We denote by $e_{k_1}e_{k_2}\cdots e_{k_{2m}}$, $e_{k_i}\in E(G)$, $1\leq k_i\leq m$, the $2m$ marking arrows along $C$ in a clockwise direction.
For the edge $e_i$, $1\leq i\leq m$, suppose $e_{k_i}$ and $e_{k_{i+j+1}}$ are two marking arrows of $e_i$ in $C$, we can divide $C$ into two disjoint arcs $D'_i$ and $D''_i$, in which, $e_{k_i},e_{k_{i+1}},\cdots,e_{k_{i+j}}$ and $e_{k_{i+j+1}},e_{k_{i+j+2}},\cdots,e_{k_{i-1}}$ lie, respectively. Denoted by
$D'_i=e_{k_i}e_{k_{i+1}}\cdots e_{k_{i+j}}$ and $D''_i=e_{k_{i+j+1}}e_{k_{i+j+2}}\cdots e_{k_{i-1}}$.
Let $|D'_i|$ and $|D''_i|$ be the number of marking arrows contained in $D'_i$ and $D''_i$, respectively.
We suppose that

$$ D_i=\left\{
\begin{array}{lc}
D'_i~or~D''_i,             & {if ~ |D'_i|=|D''_i|;}\\
D'_i,           & {if ~ |D'_i|<|D''_i|;}\\
D''_i,           & {if ~ |D'_i|>|D''_i|.}
\end{array} \right. $$
We call $D_i$ a shorter marking arrow sequence. Let $\mathcal{C}(G)=\{ D_i~|~1 \leq i \leq m\}$ be the set of all shorter marking arrow sequences of $G$ (with a fixed cyclic ordering of marking arrows along $C$).

\begin{example}\label{ex 2}
An arrow presentation of a bouquet $G_1$ with $E(G_1)=\{e_1,\cdots, e_5\}$ is shown in Figure \ref{f 18}, in which, we shall write $i$ to represent the edge $e_i$ for convenience, $1\leq i\leq 5$. $C(G_1)=e_1e_2e_4e_3e_2e_1e_3e_4e_5e_5$. Then we have $\mathcal{C}=\{ D_1,D_2,\cdots,D_5\}$, where $D_1=e_1e_2e_4e_3e_2$ or $D_1=e_1e_3e_4e_5e_5$, $D_2=e_2e_4e_3$,
$D_3=e_3e_2e_1$, $D_4=e_4e_5e_5e_1e_2$ or $D_4=e_4e_3e_2e_1e_3$, $D_5=e_5$.
\begin{figure}[htbp]
\centering
\includegraphics[width=1.3in]{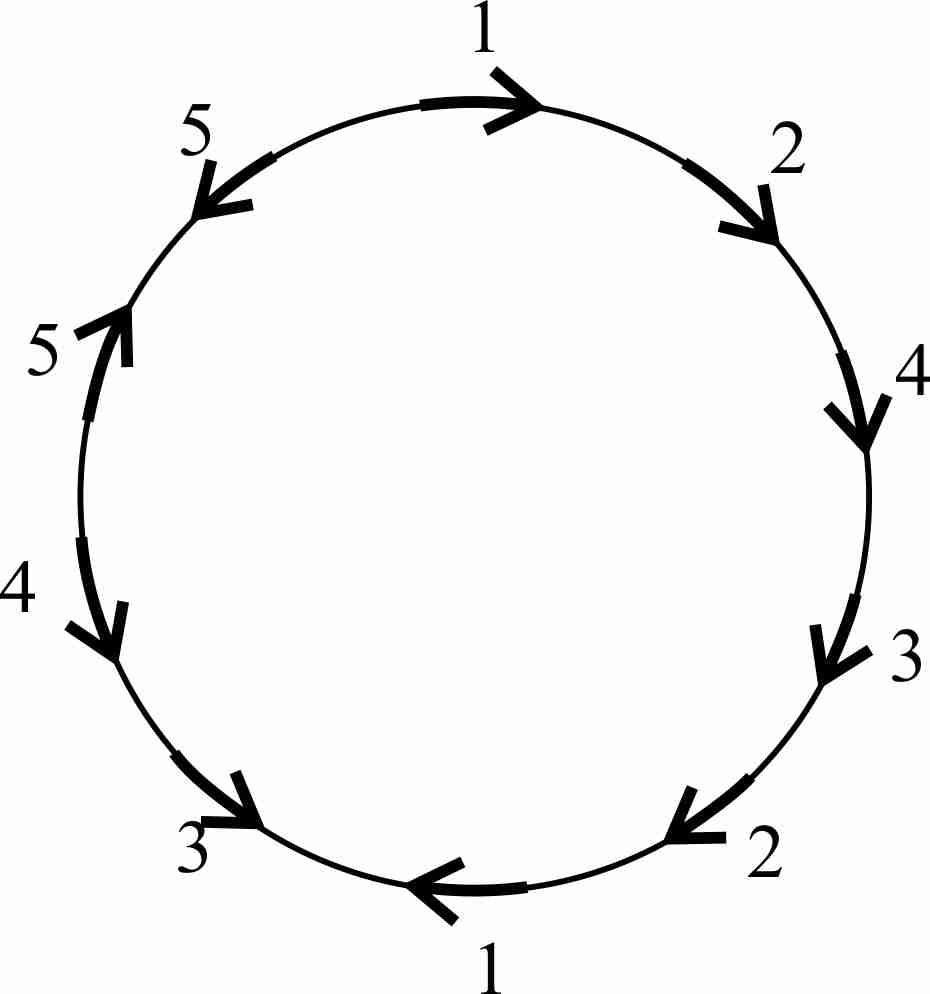}
\caption{Arrow presentation of the bouquet $G_1$.}
\label{f 18}
\end{figure}
\end{example}

\begin{definition}
Let $\{C_1, C_2, \cdots, C_n\}\subseteq \mathcal{C}(G)$. If the corresponding edge (i.e. the initial edge) of each $C_i$ $(i=1,2,\cdots,n)$ is an orientable loop and, for any pair $C_i$ and $C_j$,  $C_i\subseteq C_j$, $C_j\subseteq C_i$ or $C_i \cap C_j =\emptyset$, then we call $\{C_1, C_2, \cdots, C_n\}$ a shorter marking arrow sequence set.
\end{definition}

We take the convention that the empty set is also a shorter marking arrow sequence set.

\begin{example}\label{ex 1}
Suppose $G_2$ is a bouquet, $E(G_2)=\{e_1,\cdots, e_{10}\}$ and
$$C(G_2)=e_1e_7e_8e_{10}e_1e_{10}e_3e_4e_6e_9e_8e_4e_7e_9e_3e_2e_5e_6e_5e_2.$$
It is shown in Figure \ref{f 11}, in which $i$ represents the edge $e_i$, $1\leq i\leq 10$.
Then $\{C_1,C_2,C_3,C_4\}$ is a shorter marking arrow sequence set of $G_2$, where $C_1=e_2e_5e_6e_5$, $C_2=e_1e_7e_8e_{10}$, $C_3=e_9e_8e_4e_7$ and $C_4=e_3e_4e_6e_9e_8e_4e_7e_9$.
\begin{figure}[htbp]
\centering
\includegraphics[width=1.8in]{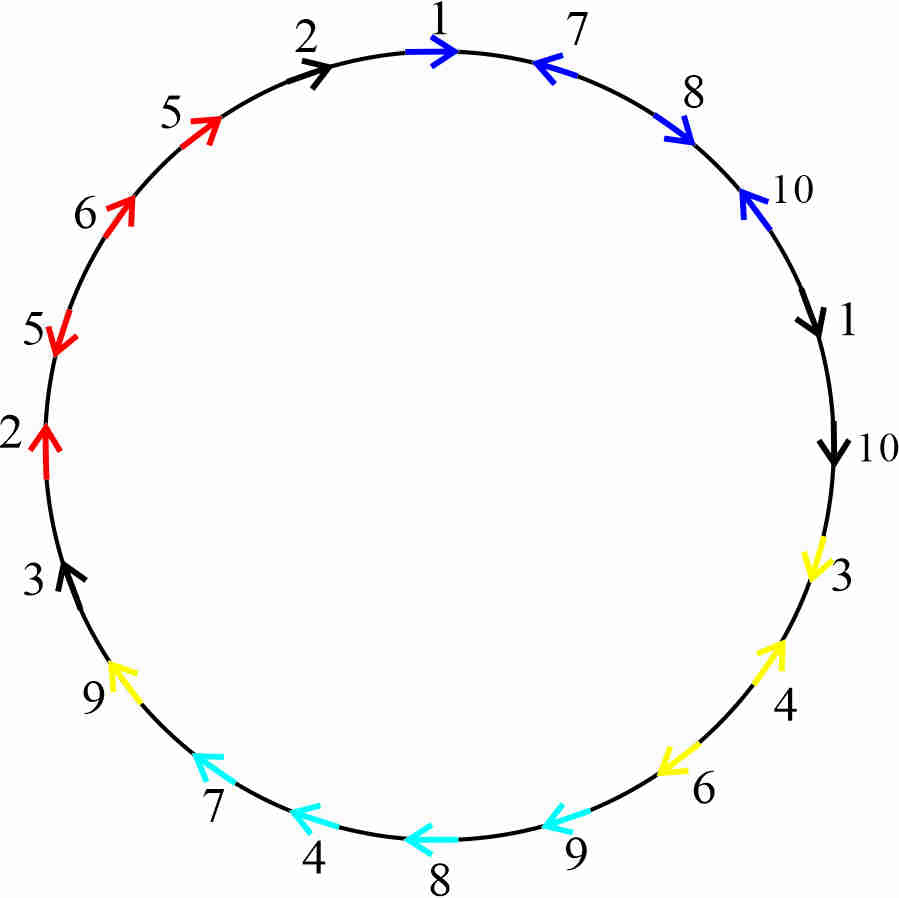}
\caption{Arrow presentation of the bouquet $G_2$.}
\label{f 11}
\end{figure}
\end{example}

\begin{definition}
Let $\{C_1,C_2,\cdots,C_n\}$ be a shorter marking arrow sequence set of the bouquet $G$.
For $1\leq k\leq n$, we define the length $d(C_k)$ of $C_k$ to be the number of marking arrows contained in
$ C_k \backslash
\mathop{\cup}\limits_{\scriptstyle 1\leq i\leq n\atop\scriptstyle i\neq k }
 C_i $.
\end{definition}

For example, in Figure \ref{f 4} (up and left), suppose that $C_1=e_1e_4e_5e_2e_6e_7e_8e_2e_9$ (colored green and red), $C_2=e_2e_6e_7e_8$ (colored red) and $C_3=e_3e_{10}e_{11}e_{12}e_{13}$ (colored blue). Then $d(C_1)=5, d(C_2)=4$, and $d(C_3)=5$.

\begin{lemma}\label{31}
With notations as above, the degree sequence of $G^{\delta(\{e_1,e_2,\cdots,e_n\})}$ is $d(C_1)$, $d(C_2)$, $\cdots$, $d(C_n)$, $2m-d(C_1)-d(C_2)-\cdots-d(C_n)$, where $e_i$ is the edge corresponding to $C_i$.
\end{lemma}

\begin{proof}
It is obvious and Figure \ref{f 4} illustrates the proof of the lemma.
\end{proof}
\begin{figure}[bhtp]
\centering
\includegraphics[width=5.5in]{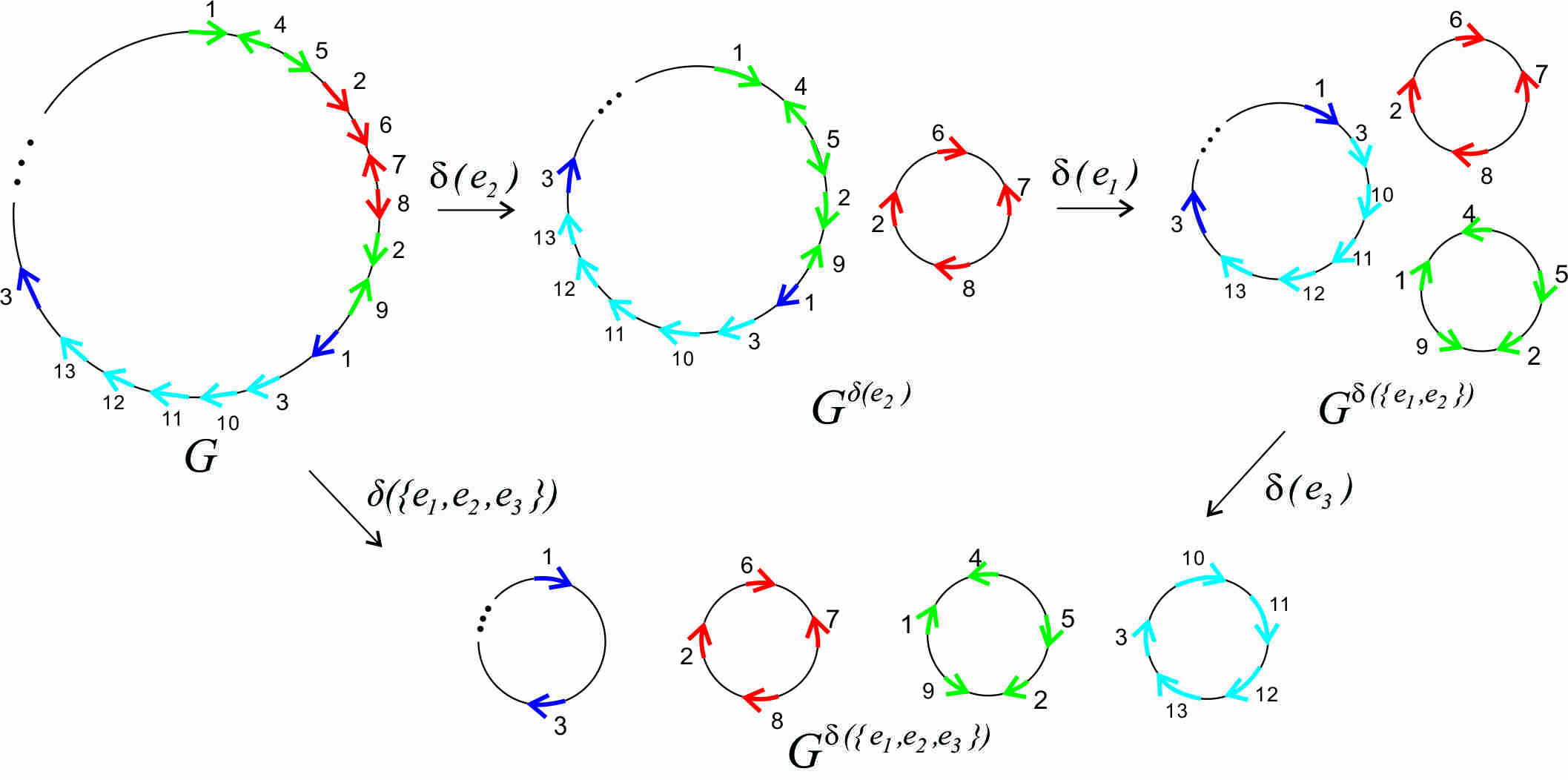}
\caption{Shorter marking arrow sequences and partial duals.}
\label{f 4}
\end{figure}

\begin{lemma}\label{32}
Let $G$ be a ribbon graph and $|E(G)|=m$. If $G^{\delta(E_1)}$ is a bouquet and $G^{\delta(E_1)\delta(E_2)}$ is $k$-regular for $E_1, E_2\subseteq E(G)$ and let $n=\frac{2m}{k}-1$, then $| E_2| \geq n$.
\end{lemma}

\begin{proof}
Let $H$ be a ribbon graph and $e\in E(H)$. It is obvious that $|V(H)|-1\leq|V(H^{\delta(e)})|\leq |V(H)|+1$. If $|E_2|<n$, then we have $|V(G^{\delta(E_1)\delta(E_2)})|\leq |V(G^{\delta(E_1)})|+|E_2|=1+|E_2|<1+n=\frac{2m}{k}$, contradicting the $k$-regularity of $G^{\delta(E_1)\delta(E_2)}$.
\end{proof}

Let $A,B$ be two sets. We denote by $A\Delta B$ the symmetric difference of $A$ and $B$.

\begin{theorem}\label{main1}
Let $G$ be a ribbon graph, $E(G)=\{e_1,\cdots, e_m\}$ and $A\subseteq E(G)$.  Then $G^{\delta(A)}$ is $k$-regular if and only if there exists a spanning quasi-tree $Q$ of $G$, a shorter marking arrow sequence set
$S=\{C_1, \cdots, C_n\}$ of $G^{\delta(E(Q))}$ with $n=\frac{2m}{k}-1$  and  $d(C_i)=k$ for $i=1,2,\cdots,n$, and
$A=E(Q) \Delta \{e_1,\cdots,e_n\}$ where $e_i$ is the edge corresponding to the shorter marking arrow sequence $C_i$ for $i=1,2,\cdots,n$.
\end{theorem}
\begin{proof} Note that $G^{\delta(A)}=G^{\delta(E(Q) \Delta \{e_1,\cdots,e_n\})}=(G^{\delta(E(Q))})^{\delta(\{e_1,\cdots,e_n\})}$. By Lemma \ref{l 6}, $G^{\delta(E(Q))}$ is a bouquet. Then the sufficiency follows from Lemma \ref{31}.

Now we prove the necessity. If $G^{\delta(A)}$ is a bouquet, by Lemma \ref{l 6}, we take $E(Q)=A$ and $S=\emptyset$. If $G^{\delta(A)}$ is not a bouquet and $k$-regular, then $2m=k|V(G^{\delta(A)})|$ and $n=\frac{2m}{k}-1\geq 1$.
Let $Q_0$ be a spanning tree of $G$ and suppose that $|A \Delta E(Q_0)|=s$. Note that $G^{\delta(E(Q_0))}$ is a bouquet. Since $G^{\delta(A)}=G^{\delta(E(Q_0))\delta(A\Delta E(Q_0))}$, by Lemma \ref{32}, we have $s\geq n$. Let $C=e_{k_1}e_{k_2}\cdots e_{k_{2m}}$, $e_{k_j}\in E(G)$ for $j=1,2,\cdots,2m$, be the cyclic marking arrow sequence in the circle of the arrow presentation of $G^{\delta(E(Q_0))}$.

\noindent{\bf Case 1.}  If $s=n$, we take $Q=Q_0$. For any edge $e\in A \Delta E(Q_0)$, $e$ must be an orientable loop of $G^{\delta(E(Q_0))}$. Otherwise, $G^{\delta(E(Q_0))\delta(e)}$ will also be a bouquet.
But $G^{\delta(A)}=(G^{\delta(E(Q_0)\Delta e)})^{\delta(A\Delta E(Q_0)\Delta e)}$ is $k$-regular and $|A\Delta E(Q_0)\Delta e)|=s-1 <  n$, contradicting Lemma \ref{32}. We next show that no two edges of $A\Delta E(Q_0)$ alternate in the bouquet $G^{\delta(E(Q_0))}$. Otherwise, suppose $e_1, e_2\in A\Delta E(Q_0)$ alternate in $G^{\delta(E(Q_0))}$, then $(G^{\delta(E(Q_0))})^{\delta(\{e_1,e_2\})}$ is also a bouquet. Similarly,
$G^{\delta(A)}=(G^{\delta(E(Q_0)\Delta\{e_1,e_2\})})^{\delta(A\Delta E(Q_0)\Delta\{e_1,e_2\})}$ is $k$-regular,
but $$|A\Delta E(Q_0)\Delta\{e_1,e_2\}|=s-2<n,$$ contradicting Lemma \ref{32}. Let $C_i$, $i=1,2,\cdots,n$, be the shorter marking arrow sequences corresponding to edges in $A \Delta E(Q_0)$. Then $S=\{C_1, \cdots, C_n\}$ is a shorter marking arrow sequence set of $G^{\delta(E(Q_0))}$ and $E(Q_0)\Delta (A\Delta E(Q_0))$ $=A$.

\noindent{\bf Case 2.} If $s> n$, we set $E_1=\emptyset$ if each edge in $A\Delta E(Q_0)$ is an orientable loop of $G^{\delta(E(Q_0))}$.
Otherwise, take a non-orientable loop $e_{p_1}\in A\Delta E(Q_0)$ of $G^{\delta(E(Q_0))}$, we then consider the edges in $A\Delta E(Q_0)\Delta{e_{p_1}}$ of $G^{\delta(E(Q_0))\delta(e_{p_1})}$, take a non-orientable loop $e_{p_2}$. Continuing the above process until we obtain a set $E_1=\{e_{p_1},\cdots,e_{p_{l}}\}\subseteq A\Delta E(Q_0)$ such that $e_{p_i}$ is a non-orientable loop of $G^{\delta(E(Q_0))\delta(\{e_{p_1},\cdots,e_{p_{i-1}}\})}$, $1\leq i\leq l$, and each edge in $A\Delta E(Q_0)\Delta E_1$ is an orientable loop of $G^{\delta(E(Q_0)\Delta E_1)}$. Then clearly $G^{\delta(E(Q_0)\Delta E_1)}$ is a bouquet.

We set $E_2=\emptyset$ if any two edges in $A\Delta E(Q_0)\Delta E_1$ are non-alternate in $G^{\delta(E(Q_0)\Delta E_1)}$.
Otherwise, take two alternate edges $e_{q_1}, e_{q_2}\in A\Delta E(Q_0)\Delta E_1$ of $G^{\delta(E(Q_0)\Delta E_1)}$, then consider edges in $A\Delta E(Q_0)\Delta E_1\Delta\{e_{q_1}, e_{q_2}\}$ of    the ribbon graph $G^{\delta(E(Q_0)\Delta E_1\Delta\{e_{q_1}, e_{q_2}\})}$ and take two alternate edges $e_{q_3}, e_{q_4}$. Continuing this process until we obtain $E_2=\{e_{q_1},\cdots,e_{q_{2j}}\}\subseteq A\Delta E(Q_0)\Delta E_1$ such that $e_{q_i}$ and $e_{q_{i+1}}$ are alternate in $G^{\delta(E(Q_0)\Delta E_1\Delta \{e_{q_1},\cdots,e_{q_{i-1}}\})}$, $i=1,3,5,\cdots, 2j-1$, and no two edges in $A\Delta E(Q_0)\Delta E_1\Delta E_2$ are alternate in $G^{\delta(E(Q_0)\Delta E_1\Delta E_2)}$.
Then $G^{\delta(E(Q_0)\Delta E_1\Delta E_2)}$ is still a bouquet.

Recall that $G^{\delta(E(Q_0)\Delta E_1)}[A\Delta E(Q_0)\Delta E_1]$ is an orientable ribbon graph,  by Lemma \ref{l 4}, $G^{\delta(E(Q_0)\Delta E_1\Delta (E_2))}[A\Delta E(Q_0)\Delta E_1]$ is also orientable.
Thus, each edge in $A\Delta E(Q_0)\Delta E_1\Delta E_2$ ($E_2\subseteq A\Delta E(Q_0)\Delta E_1$) is an orientable loop of $G^{\delta(E(Q_0)\Delta E_1\Delta E_2)}$.

Take $E(Q)=E(Q_0)\Delta E_1\Delta E_2$, then $Q$ is a spanning quasi-tree of $G$ by Lemma \ref{l 6}. Note that $A=(E(Q_0)\Delta E_1\Delta E_2)\Delta (A\Delta E(Q_0)\Delta E_1\Delta E_2)$. Let $S$ be the set of shorter marking arrow sequences corresponding edges in $A\Delta E(Q_0)\Delta E_1\Delta E_2$. Then $S$ is a shorter marking arrow sequence set of $G^{\delta(E(Q))}$ whose each shorter marking arrow sequence has length $k$. Furthermore,
\begin{eqnarray*}
& &|A\Delta E(Q_0)\Delta E_1\Delta E_2 |\\
&=&|V((G^{\delta(E(Q_0)\Delta E_1\Delta E_2)})^{\delta(A\Delta E(Q_0)\Delta E_1\Delta E_2)})|-1\\
&=&|V(G^{\delta(A)})|-1\\
&=&\frac{2m}{k}-1\\
&=&n.
\end{eqnarray*}
This completes the proof of Theorem \ref{main1}.
\end{proof}

\section{The checkerboard colourable partial Petrials}
\noindent

It is easy to show that any Eulerian ribbon graph has a checkerboard colourable partial Petrial \cite{YJ11}. In this section we try to characterize all checkerboard colourable partial Petrials for an Eulerian ribbon graph.

\begin{lemma}\label{l 8}
Let $G$ be a ribbon graph and $e\in E(G)$. Then $G$ is checkerboard colourable implies that $G/e$ is checkerboard colourable.
\end{lemma}
\begin{proof}
Let $e'$ and $e''$ be the two edge line segments of the edge $e$ of $G$. Then $e'$ (resp. $e''$) will connect two incident (not necessarily distinct) vertex line segments $v_1'$ and $v_2'$ (resp. $v_1''$ and $v_2''$) of $G$ to form a vertex line segment $v'=v_1'e'v_2'$ (resp. $v''=v_1''e''v_2''$) (called the corresponding vertex line segment in the following) of $G/e$.  If $G$ has been checkerboard colored, then $v_1'$, $e'$ and $v_2'$ (resp. $v_1''$, $e''$ and $v_2''$) receive the same color (but $v'$ and $v''$ receive different colors), naturally forming a checkerboard coloring of $G/e$. See Figure \ref{f 6}.
\begin{figure}[htbp]
\centering
\includegraphics[width=2.5in]{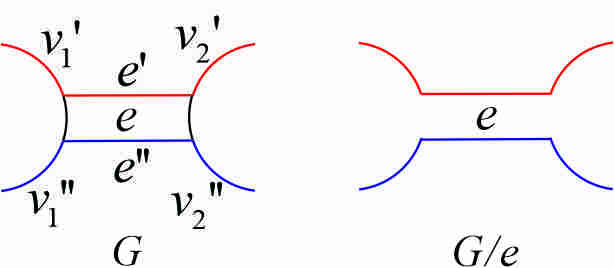}
\caption{The checkerboard colouring of $G/e$ inheriting from a checkerboard colouring of $G$.}
\label{f 6}
\end{figure}
\end{proof}

\begin{lemma}\label{l 9}
Let $G$ be an Eulerian ribbon graph and $F$ be a forest of $G$. Then $G/E(F)$ is checkerboard colourable implies that $G$ is checkerboard colourable.
\end{lemma}
\begin{proof}
Suppose that $G/E(F)$ has been checkerboard coloured.
If two corresponding vertex line segments of each edge of $F$ in $G/E(F)$ are coloured with different colours, then by reversing the process in the proof of Lemma \ref{l 8}, we obtain a checkerboard coloring of $G$.

Suppose $E(F)=\{e_1,\cdots,e_n\}$. Since $F$ is a forest, each connected component of $G[F]$ forms a vertex of $G/E(F)$. In particular two corresponding vertex line segments of each edge in $F$ belong to the same vertex boundary of $G/E(F)$.

If there exists an edge $e_k\in E(F)$, $1\leq k\leq n$, such that the corresponding vertex line segments of $e_k$ are coloured with same colour, implying that there are even number of half-edges between these two vertex line segments in $G/E(F)$,
then, in $G/(E(F)\backslash e_k)$, both the degree of the two ends of the edge $e_k$, denoted by $u_1$ and $u_2$, will be odd. Suppose $u'_1$ and $u'_2$ are the corresponding vertices of $u_1$ and $u_2$ in $G^{\delta(E(F)\backslash e_k)}$. Clearly, both the degree of $u'_1$ and $u'_2$ are still odd and the edges of $E(F)\backslash e_k$ are orientable and pairwise non-alternate at $u'_1$ and $u'_2$.

Note that $G$ can be obtained form $G^{\delta(E(F)\backslash e_k)}$ by taking partial dual with respect to $E(F)\backslash e_k$. Since $u'_1$ and $u'_2$ are odd vertices in $G^{\delta(E(F)\backslash e_k)}$, there must exist two odd vertices in $G$, contradicting that $G$ is Eulerian.
\end{proof}

\begin{remark}

In Lemma \ref{l 9}, the condition that $G$ is an Eulerian is necessary. For example, let $G$ be an orientable ribbon graph which is a $3$-cycle with a pendant edge. Then $G/E(T)$ is a bouquet with an orientable loop for any spanning tree $T$ of $G$ which is checkerboard colourable, but $G$ is not checkerboard colorable.

In particular, Lemma \ref{l 9} holds when $F$ is a spanning tree of $G$. However, it is possibly wrong when $F$ is spanning quasi-tree of $G$. An example is given in Figure \ref{f 9}, where $\{e_1,e_3\}$ forms a spanning quasi-tree of $G$, $G/\{e_1,e_3\}$ is checkerboard colourable, but $G$ is not checkerboard colorable.

\begin{figure}[htbp]
\centering
\includegraphics[width=5.0in]{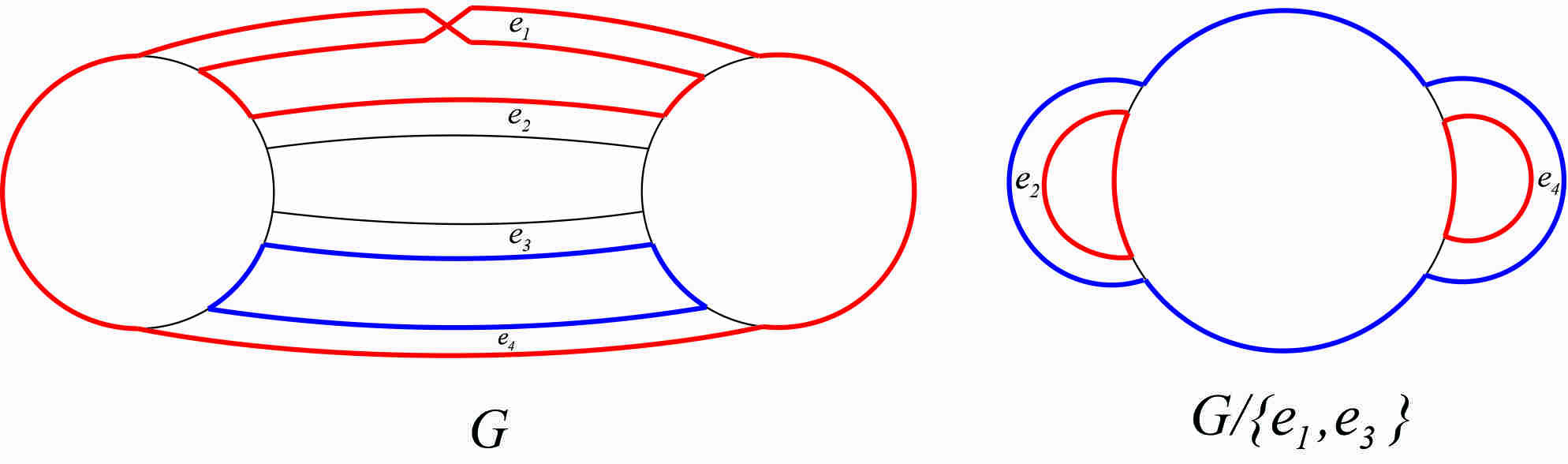}
\caption{$\{e_1,e_3\}$ forms a spanning quasi-tree but not a spanning tree.}
\label{f 9}
\end{figure}

\end{remark}
Combining Lemmas \ref{l 8} and \ref{l 9}, we obtain:

\begin{theorem}\label{th 5}
An Eulerian ribbon graph $G$ is checkerboard colourable if and only if $G/E(F)$ is checkerboard colourable for any forest $F$ of $G$.
\end{theorem}

Let $G$ be a ribbon graph. The arrow presentation of $G$ can be viewed as the union of all vertex line segments of $G$ and making arrows of all edges of $G$ with vertex line segments and marking arrows appearing alternatively in each circle of the arrow presentation.
Suppose $e', e''$ are two marking arrows of $e\in E(G)$ in the arrow presentation of $G$. Then $G$ is checkerboard colourable if and only if there exists a coloring of vertex line segments with two colors in its arrow presentation such that
\begin{enumerate}
\item vertex line segments are colored with two colors alternatively in each circle, or equivalently the head and tail (to be exact, the vertex line segments incident with the head and the tail) of each marking arrow receive different colors.
\item the color of the head of $e'$ is same as the tail of $e''$, or equivalently (if the condition (1) holds) the color of the tail of $e'$ is same as the head of $e''$, for each $e\in E(G)$.
\end{enumerate}

It is obvious that if a ribbon graph is checkerboard colorable, then it must be Eulerian by (1).

Let $G$ be a ribbon graph and $T$ be a spanning tree of $G$. Then $G/E(T)$ is a bouquet. If we color vertex line segments of the arrow presentation of $G/E(T)$ alternatively (this can be done), then condition (1) is satisfied. Then the condition (2) will be equivalent to that the number of vertex line segments from the vertex line segment incident with the head of $e'$ to that incident with the tail of $e''$ (briefly, the number of vertex line segments between the head of $e'$ and the tail of $e''$) is odd in $G/E(T)$ for each $e\in E(G)\backslash E(T)$.

\begin{definition}
Let $G$ be a ribbon graph and $T$ be a spanning tree of $G$. For $e\in E(G)\backslash E(T)$, we define the index $t(e,G/E(T))=0$ if the head of $e'$ and the tail of $e''$ are separated by odd number of vertex line segments in the arrow presentation of $G/E(T)$ and $t(e,G/E(T))=1$ otherwise.
\end{definition}

An example is given in Figure \ref{e10} with $t(e_1,G/E(T))=0$ and $t(e_2,G/E(T))=0$.

\begin{figure}[htbp]
\centering
\includegraphics[width=1.5in]{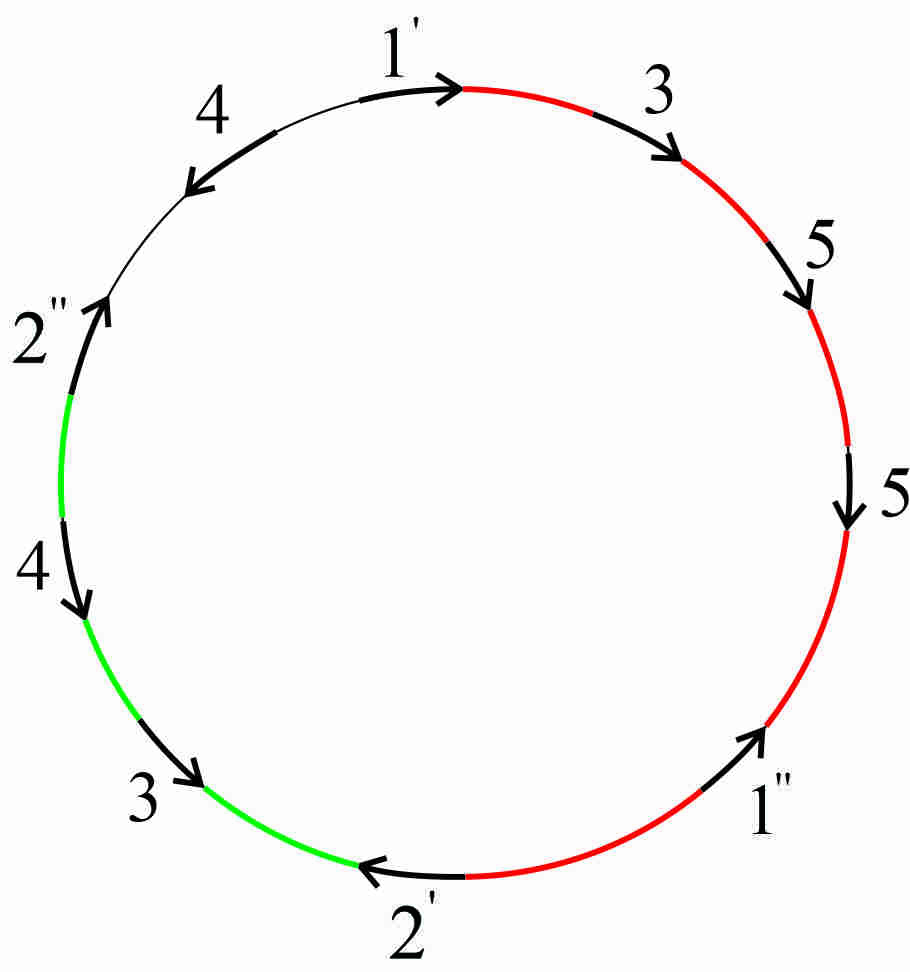}
\caption{The arrow presentation of $G/E(T)$.}
\label{e10}
\end{figure}

Evidently, we have:

\begin{lemma}\label{k11}
Let $G$ be a ribbon graph and $T$ be a spanning tree of $G$. Then $G$ is checkerboard colourable if and only if $t(e,G/E(T))=0$ for each $e\in E(G)\backslash E(T)$.
\end{lemma}

In order to label positions of the corresponding vertex line segments (i.e. vertex line segments corresponding to edge line segments) of $E(T)$ in the bouquet $G/E(T)$, we add a short vertical line segment to the head of each marking arrow of each edge of $E(T)$ in the arrow presentation of $G^{\delta(E(T))}$ instead of deleting them as in the arrow presentation of $G/E(T)$. We denoted by $\overline{G/E(T)}$ such an arrow presentation with short vertex line segments added to heads of marking arrows of edges of $E(T)$ and we call the marking arrow with head added short vertex line segment \emph{vertical marking arrow}. An example is given in Figure \ref{f 22}.

These vertical marking arrows only indicate the positions and we still consider them as parts of vertex line segments but not the marking arrows, thus $t(e,G/E(T))=t(e,\overline{G/E(T)})$ for each $e\in E(G)\backslash E(T)$.
When $e\in E(T)$, suppose $\overline{e}'$ and $\overline{e}''$ are the vertical marking arrows of $e$.
There are two marking arrow sequences between $\overline{e}'$ and $\overline{e}''$ in $\overline{G/E(T)}$, then we randomly choose one of them and denote by $P_e$.

\begin{definition}
As notations above, for each $e\in E(T)$, let $E_e=\{e\}\cup\{f\in E(G)\backslash E(T)~| $ only one marking arrow of $f$ is included in $P_e \}$, we call $E_e$ the adjoint set of $e$.
\end{definition}

\begin{lemma}\label{l 10}
Let $G$ be a ribbon graph, $T$ be a spanning tree of $G$, $e\in E(T)$ and $E_e$ be the adjoint set of $e$. Then $G$ is checkerboard colourable implies that $G^{\tau(E_e)}$ is checkerboard colourable.
\end{lemma}
\begin{proof}
Since $G$ is checkerboard colourable, the ribbon graph $G/E(T)$ is checkerboard colourable by Lemma \ref{l 8}. Then, by Lemma \ref{k11}, $t(f,G/E(T))=0$ for $f\in E(G)\backslash E(T)$. There are two cases. See Figure \ref{f 22}.

\begin{enumerate}
\item If both or neither the two marking arrows of $f\in E(G)\backslash E(T)$ are contained in $P_e$, then the parity of the number of marking arrows between the head of $f'$ and the tail of $f''$ in $\overline{G/E(T)}$ is same as that in $\overline{G^{\tau(e)}/E(T)}$. Thus, $t(f,G^{\tau(e)}/E(T))=t(f,\overline{G^{\tau(e)}/E(T)})=t(f,\overline{G/E(T)})$ $=t(f,G/E(T))=0$.
\item As in the proof of Lemma \ref{l 9}, the two vertex line segments including $\overline{e}'$ and $\overline{e}''$ are coloured with two different colours. Thus the number of the vertex line segments between $\overline{e}'$ and $\overline{e}''$ are even. If only one marking arrow of $f\in E(G)\backslash E(T)$ is contained in $P_e$,
then the sum of the number of vertex line segments between the head of $f'$ and the tail of $f''$ in
 $\overline{G/E(T)}$ and the number of vertex line segments between the head of $f'$ and the head $f''$ in $\overline{G^{\tau(e)}/E(T)}$ is even.
Therefore, the number of vertex line segments between the head of $f'$ and the tail of $f''$ in $\overline{G/E(T)}$ is odd if and only if the number of vertex line segments between the head of $f'$ and the tail of $f''$ in $\overline{G^{\tau(e)}/E(T)}$ is even.
Due to $t(f,\overline{G/E(T)})=t(f,G/E(T))=0$, we have $t(f,G^{\tau(e)}/E(T))=t(f,\overline{G^{\tau(e)}/E(T)})=1$.
\end{enumerate}

\begin{figure}[htbp]
\centering
\includegraphics[width=3.5in]{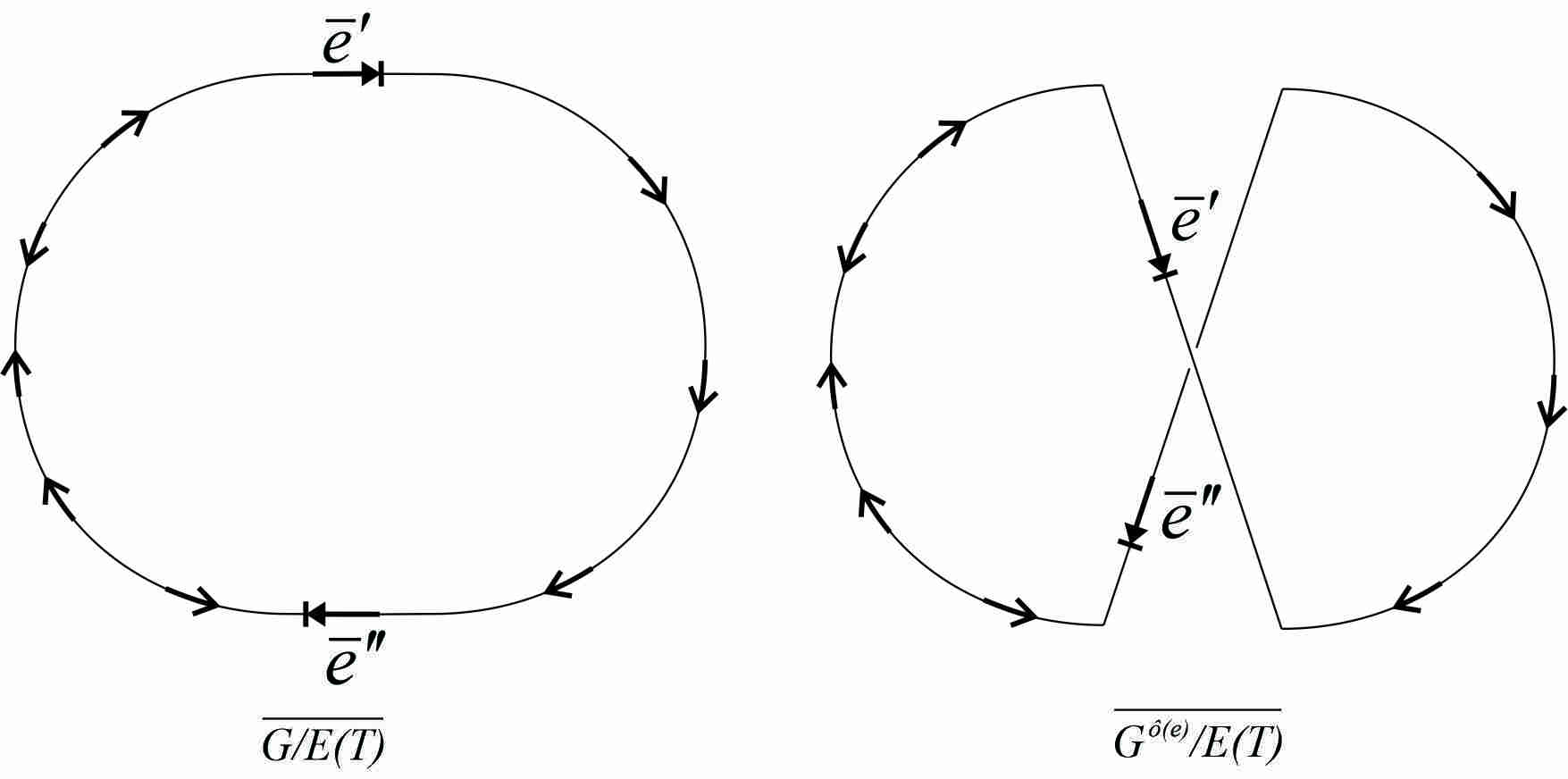}
\caption{$\overline{G/E(T)}$ and $\overline{G^{\tau(e)}/E(T)}$, $e\in E(T)$.}
\label{f 22}
\end{figure}

Thus by Lemma \ref{k11} we have $(G^{\tau(e)}/E(T))^{\tau(E_e\backslash e)}$ is checkerboard colourable.
Clearly, $G^{\tau(f)}/E(T)=(G/E(T))^{\tau(f)}$ for any $f\in E(G)\backslash E(T)$.
Therefore, $G^{\tau(E_e)}/E(T)=(G^{\tau(e)}/E(T))^{\tau(E_e\backslash e)}$ is checkerboard colourable.

When $G$ is checkerboard colourable, $G$ is an Eulerian ribbon graph.
Note that $T$ is also a spanning tree of $G^{\tau(E_e)}$. By Theorem \ref{th 5}, $G^{\tau(E_e)}$ is checkerboard colourable.
\end{proof}

\begin{theorem}\label{th 2}
Let $G$ be an Eulerian ribbon graph and $T$ be a spanning tree of $G$. Then
$G^{\tau(A)}$ is checkerboard colourable for $A\subseteq E(G)$ if and only if $A\cap E(T)=\{e_1,\cdots,e_s\}$ and
$$A=A_0\Delta E_{e_1}\Delta \cdots\Delta E_{e_s},$$
where $A_0=\{e\in E(G)\backslash E(T)~| ~t(e,G/E(T))=1\}$ and $E_{e_i}$ is the adjoint set of $e_i$, $1\leq i\leq s$.
\end{theorem}

\begin{proof}
$(\Longleftarrow)$
Clearly, by Lemma \ref{k11} the ribbon graph $(G/E(T))^{\tau(A_0)}$ is checkerboard colourable.
Since $(G/E(T))^{\tau(A_0)}=G^{\tau(A_0)}/E(T)$, $G^{\tau(A_0 )}$ is checkerboard
colourable by Theorem \ref{th 5}.
Then, by Lemma \ref{l 10}, $G^{\tau(A)}$ is checkerboard colourable.

$(\Longrightarrow)$ We prove the necessity by induction on $|A\cap E(T)|$. When $|A\cap E(T)|=0$, since $G^{\tau(A)}/E(T)=(G/E(T))^{\tau(A)}$, the following statements are all equivalent:
\begin{enumerate}
\item $G^{\tau(A)}$ is checkerboard colourable;
\item $G^{\tau(A)}/E(T)$ is checkerboard colourable;
\item ${(G/E(T))}^{\tau(A)}$ is checkerboard colourable;
\item $t(e, (G/E(T))^{\tau(A)})=0$ for $e\in E(G)\backslash E(T)$;
\item $t(e,G/E(T))=1$ for $e \in A$ and $t(f,G/E(T))=0$ for $f \in E(G)\backslash E(T)\backslash A$.
\end{enumerate}
Thus we can take $A$ to be the set $A_0$.

Suppose that $A\cap E(T)=\{e_1,\cdots,e_s\}$, $s\geq 1$. Since $G^{\tau(A)}$ is checkerboard colourable, then by Lemma \ref{l 10} $G^{\tau(A)\tau(E_{e_s})}$ is checkerboard colourable. Hence $G^{\tau(A\Delta E_{e_s})}=G^{\tau(A)\tau(E_{e_s})}$ is checkerboard colourable. Note that $(A\Delta E_{e_s})\cap E(T)=\{e_1,\cdots,e_{s-1}\}$. By induction hypothesis, we have
$$
A\Delta E_{e_s}=A_0\Delta E_{e_1}\Delta \cdots\Delta E_{e_{s-1}}.
$$
Thus, $A=A_0\Delta E_{e_1}\Delta \cdots\Delta E_{e_{s-1}}\Delta E_{e_s}$.
\end{proof}

\begin{remark}
Theorem \ref{th 2} is independent of the choice of the spanning tree $T$.
\end{remark}

\begin{corollary}\label{sss}
Let $G$ be an Eulerian ribbon graph. Then $G$ has $2^{|V(G)|-1}$ checkerboard colorable partial petrials (some of them may be equivalent).
\end{corollary}

\begin{proof}
Let $T$ be a spanning tree of $G$. Note that for each subset of $E(T)$, there is a unique $A$ such that $G^{\tau(A)}$ is checkerboard colourable.
\end{proof}

\begin{example}
A 4-regular ribbon graph $G$ is shown in Figure \ref{f 10} (left). Suppose that $T$ is a spanning tree of $G$ with $E(T)=\{e_1,e_2,e_3,e_4\}$. $\overline{G/E(T)}$ is shown in Figure \ref{f 10} (right).
By Figure \ref{f 10} (right), we have $E_1=\{e_1,e_7,e_8,e_{10}\}$, $E_2=\{e_2,e_6\}$, $E_3=\{e_3,e_6,e_7,e_8\}$ and $E_4=\{e_4,e_6,e_8,e_9\}$ and $A_0=\{e_5,e_6,e_9,e_{10}\}$ as listed in Table \ref{ta1}.
Then $G^{\tau(A)}$ is checkerboard colourable for $A\subseteq E(G)$ if and only if $A$ is one of the following subsets listed in Table \ref{ta11} (left).

\begin{figure}[ht]
\centering
\includegraphics[width=5.1in]{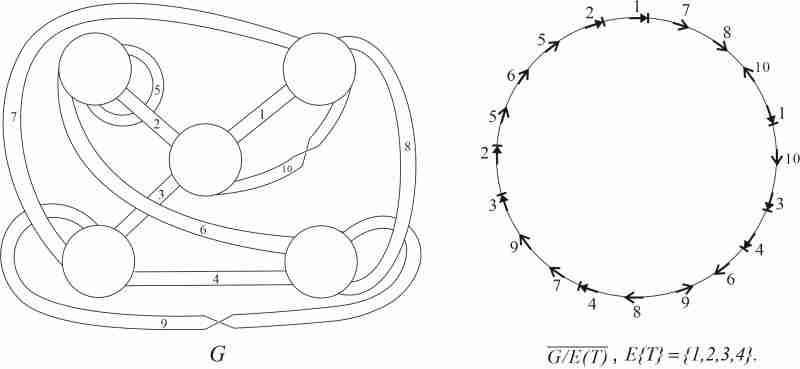}
\caption{An example.}
\label{f 10}
\end{figure}

\begin{table}[htpb]
\centering
\begin{tabular}{|c|c|c|}\hline
  edge    &   \# of vertex line segments &  $t$  \\ \hline
  $e_5$&    $2$  &  $1$                          \\ \hline
  $e_6$&    $6$  &  $1$                          \\ \hline
  $e_7$&    $5$  &  $0$                          \\ \hline
  $e_8$&    $5$  &  $0$                          \\  \hline
  $e_9$&    $4$  &  $1$                          \\ \hline
  $e_{10}$& $2$  &  $1$                          \\  \hline
\end{tabular}
\caption{The index $t$.}\label{ta1}
\end{table}

\begin{table}[htp]
\centering
\begin{tabular}{|c|c|}\hline
 the edge subset $A$ of $E(G)$     &  the corresponding subset $A\cap E(T)$ of $E(T)$  \\ \hline
  $\{e_5,e_6,e_9,e_{10}\}$&    $\emptyset$                           \\ \hline
  $\{e_1,e_5,e_6,e_7,e_8,e_9\}$&    $\{e_1\}$                           \\ \hline
  $\{e_2,e_5,e_9,e_{10}\}$&    $\{e_2\}$                           \\ \hline
  $\{e_3,e_5,e_7,e_8,e_9,e_{10}\}$&    $\{e_3\}$                           \\  \hline
  $\{e_4,e_5,e_8,e_{10}\}$&    $\{e_4\}$                           \\ \hline
  $\{e_1,e_2,e_5,e_7,e_8,e_9\}$&    $\{e_1,e_2\}$                           \\ \hline
  $\{e_1,e_3,e_5,e_9\}$&    $\{e_1,e_3\}$                           \\ \hline
  $\{e_1,e_4,e_5,e_7\}$&    $\{e_1,e_4\}$                           \\  \hline
  $\{e_2,e_3,e_5,e_6,e_7,e_8,e_9,e_{10}\}$&    $\{e_2,e_3\}$                           \\ \hline
  $\{e_2,e_4,e_5,e_6,e_8,e_{10}\}$&    $\{e_2,e_4\}$                           \\ \hline
  $\{e_3,e_4,e_5,e_6,e_7,e_{10}\}$&    $\{e_3,e_4\}$                           \\ \hline
  $\{e_2,e_3,e_4,e_5,e_7,e_{10}\}$&    $\{e_2,e_3,e_4\}$                           \\  \hline
  $\{e_1,e_3,e_4,e_5,e_6,e_8\}$&    $\{e_1,e_3,e_4\}$                           \\ \hline
  $\{e_1,e_2,e_4,e_5,e_6,e_7\}$&    $\{e_1,e_2,e_4\}$                           \\ \hline
  $\{e_1,e_2,e_3,e_5,e_6,e_9\}$&    $\{e_1,e_2,e_3\}$                           \\ \hline
  $\{e_1,e_2,e_3,e_4,e_5,e_8\}$&    $\{e_1,e_2,e_3,e_4\}$                           \\  \hline
  \end{tabular}
\caption{$A$ and $A\cap E(T)$.}\label{ta11}
\end{table}

\end{example}

\section{Ellis-Monaghan and Moffatt's problems}
\noindent

Ellis-Monaghan and Moffatt's two problems are first solved in \cite{YJ11}. We answer them as consequences of Theorem \ref{th 2}.

\begin{corollary}\label{51}
Any Eulerian ribbon graph has a checkerboard colourable partial Petrial.
\end{corollary}

\begin{proof}
It follows immediately from Corollary \ref{sss}.
\end{proof}

\begin{corollary}\label{52}
Any ribbon graph has a checkerboard colourable twisted dual.
\end{corollary}

\begin{proof}
Let $G$ be a ribbon graph. Without loss of generality, we assume that $G$ is connected. If $G$ is Eulerian, then Corollary \ref{52} follows from Corollary \ref{51}. If $G$ is not Eulerian, let $T$ be a spanning tree, then $G^{\delta(E(T))}$ is a bouquet, hence Eulerian. By Corollary \ref{51}, Corollary \ref{52} also holds.
\end{proof}

In the following, given a ribbon graph, we give a complete characterization of its all $k$-regular checkerboard colorable twisted duals.

\begin{theorem}\label{th 3}
Let $G$ be a ribbon graph, $|E(G)|=m$ and $k$ be an even positive integer. Then every $k$-regular checkerboard colourable twisted dual of $G$ can be written as the form $G^{{\tau(A_1)}{\delta(A_2)}{\tau(A_3)}}$, where
\begin{enumerate}
\item $A_1\subseteq E(G)$;
\item $A_2=E(Q) \Delta \{e_1,\cdots,e_n\}$, where $Q$ is a spanning quasi-tree of $G^{\tau(A_1)}$ and
$\{C_1, \cdots, C_n\}$, $n=\frac{2m}{k}-1$, is a shorter marking arrow sequence set of $G^{\tau(A_1)\delta(E(Q))}$ such that $d(C_i)=k$ and $e_i$ is the edge corresponding to the shorter marking arrow sequence $C_i$ for $i=1,2,\cdots,n$;
\item $A_3=A_0 \Delta E_{e_1}\Delta \cdots\Delta E_{e_s}$,
where $$A_0=\{e\in E(G)\backslash E(T)~|~ t(e,G^{\tau(A_1)\delta(A_2)}/E(T))=1 \},$$ $T$ is a spanning tree of $G^{\tau(A_1)\delta(A_2)}$, $A_3\cap E(T)=\{e_1,\cdots,e_s\}$ and $E_{e_i}$ is the adjoint set of $e_i$, $1\leq i\leq s$.
\end{enumerate}
\end{theorem}

\begin{proof}
This theorem follows from Lemma \ref{l 2} and Theorems \ref{main1} and \ref{th 2}.
\end{proof}

We conclude this paper with the following example.

\begin{example}
Let $G$ be the ribbon graph shown in Figure \ref{f 10} (left). Then $m=10$. We try to find its $4$-regular checkerboard colourable twisted duals. Then $n={2m\over k}-1=4$.

\noindent{\bf Step 1.} We randomly take $A_1=\{e_9,e_{10}\}$ and consider $G^{\tau(A_1)}$ as shown in Figure \ref{e11}.

\begin{figure}[htbp]
\centering
\includegraphics[width=2.5in]{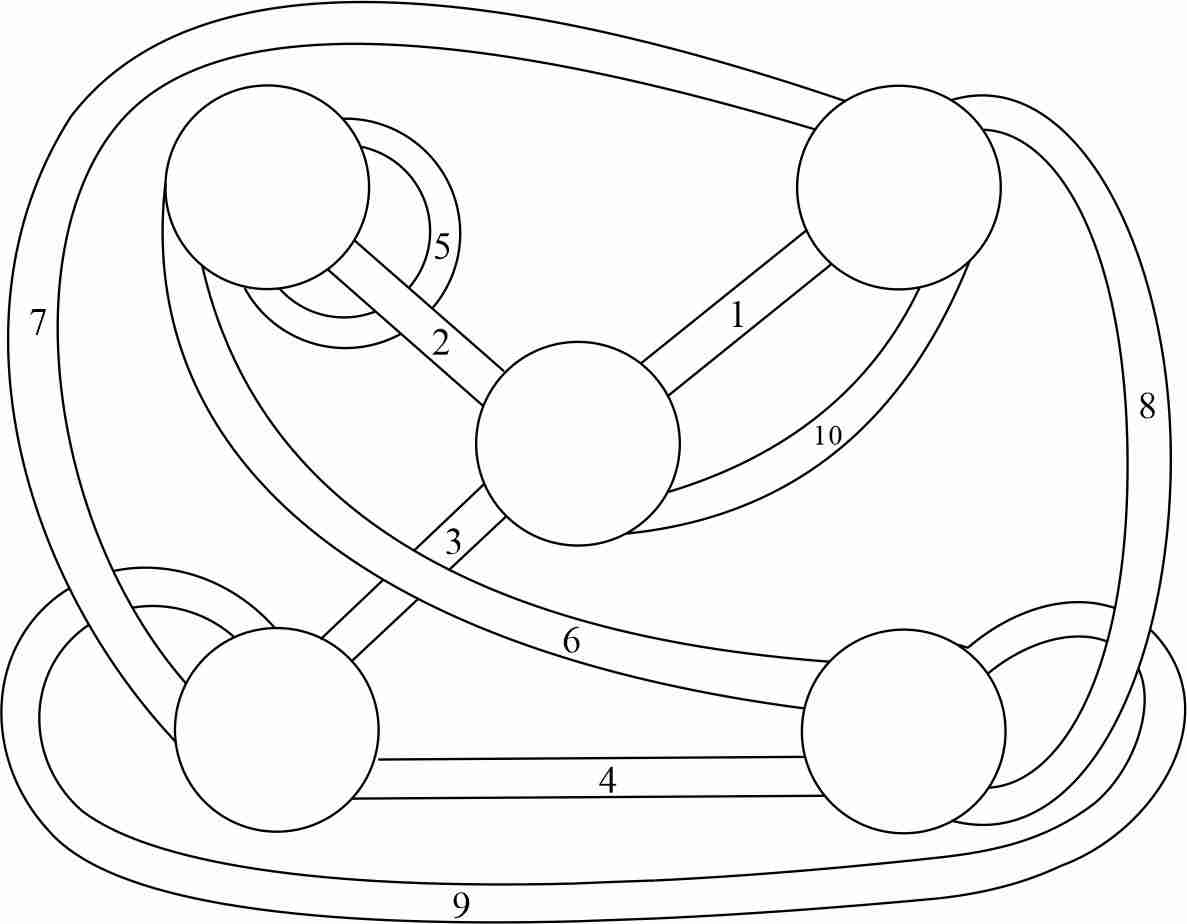}
\caption{$G^{\tau(A_1)}$.}
\label{e11}
\end{figure}
\noindent{\bf Step 2.} Now take a spanning quasi-tree $Q$ of $G^{\tau(A_1)}$ with $E(Q)=\{e_1,e_2,e_3,e_4\}$ as shown in Figure \ref{e22}. In Figure \ref{e33}, we give the arrow presentation of the bouquet $G^{\tau(\{e_9,e_{10}\})\delta(\{e_1,e_2,e_3,e_4\})}$. Then, in $G^{\tau(A_1)\delta(E(Q))}$, $\{C_1, C_2, C_3, C_4\}$ is a shorter marking arrow sequence set with $d(C_i)=4$ for $i=1,2,3,4$, where $C_1=e_1e_7e_8e_{10}$,  $C_2=e_2e_5e_6e_{5}$, $C_3=e_3e_4e_6e_9e_8e_4e_7e_9$ and $C_4=e_9e_8e_4e_7$. Thus,
$$A_2=\{e_1,e_2,e_3,e_4\}\Delta\{e_1,e_2,e_3,e_9\}=\{e_4,e_9\}.$$

\begin{figure}[htbp]
\centering
\includegraphics[width=2in]{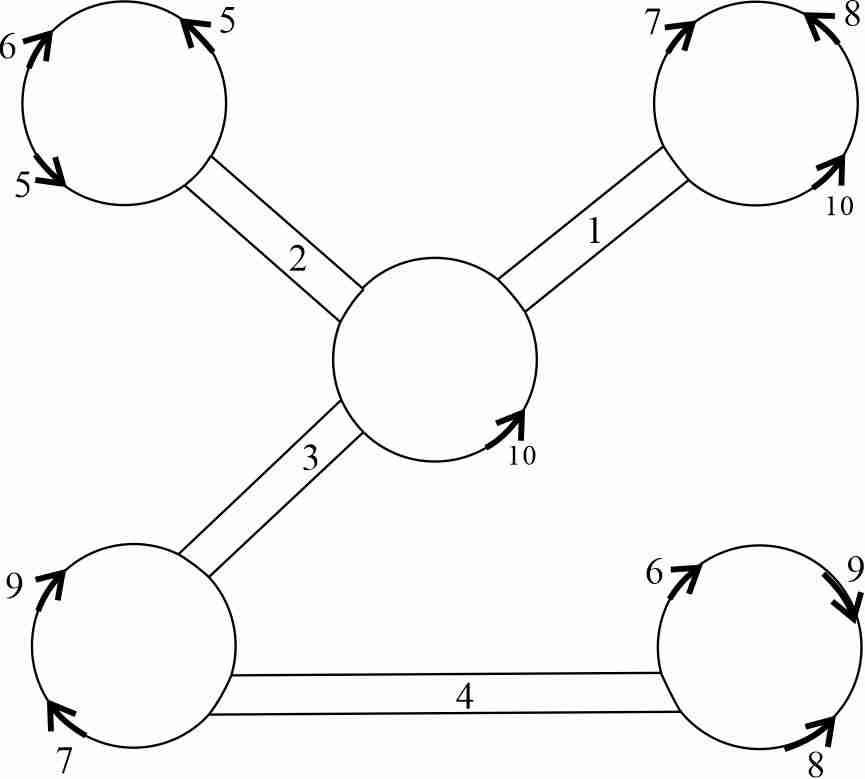}
\caption{A spanning quasi-tree $Q$ of $G^{\tau(A_1)}$.}
\label{e22}
\end{figure}

\begin{figure}[htbp]
\centering
\includegraphics[width=2.3in]{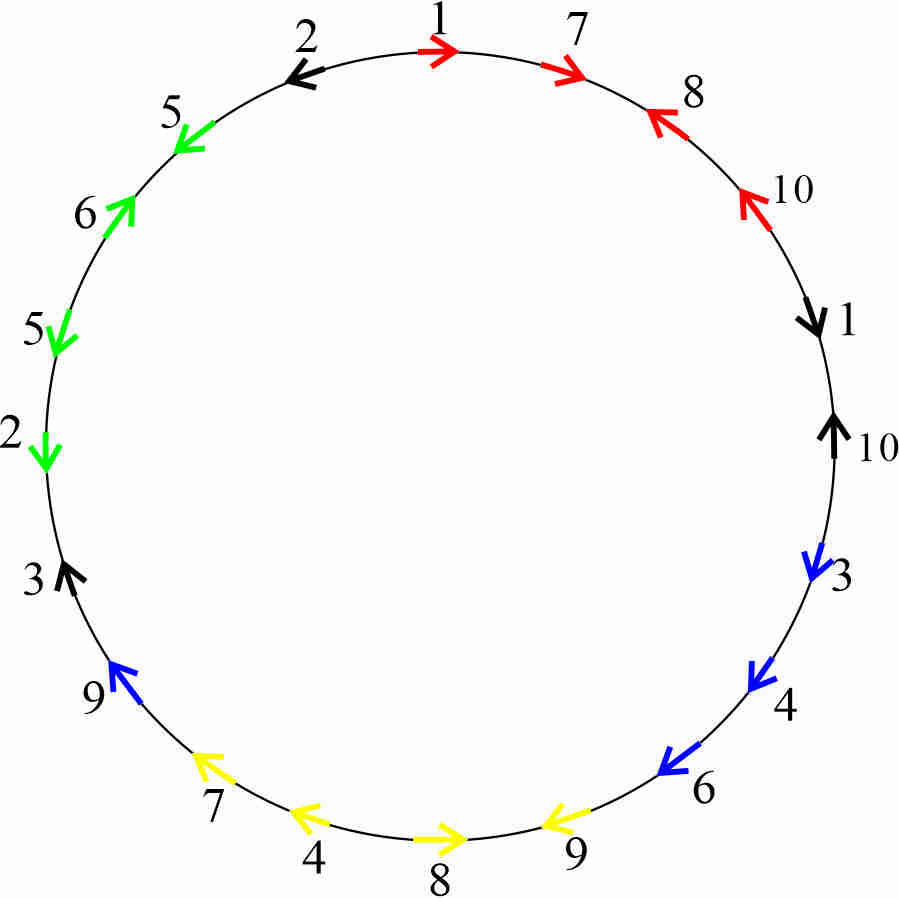}
\caption{The arrow presentation of $G^{\tau(A_1)\delta(E(Q))}$.}
\label{e33}
\end{figure}

\noindent{\bf Step 3.} Now we consider $G^{\tau(A_1)\delta(A_2)}$ as shown in Figure \ref{e55} which is obtained from Figure \ref{e11} via Figure \ref{e44}. Take a spanning tree $T$ of $G^{\tau(A_1)\delta(A_2)}$ with $E(T)=\{e_1,e_2,e_3,e_4\}$. $\overline{G^{\tau(A_1)\delta(A_2)}/E(T)}$ is shown in Figure \ref{e77} via Figure \ref{e66}. Then we obtain $A_3=\{e_5,e_7\}$ as listed in Table \ref{ta}.

\begin{figure}[htbp]
\centering
\includegraphics[width=2.5in]{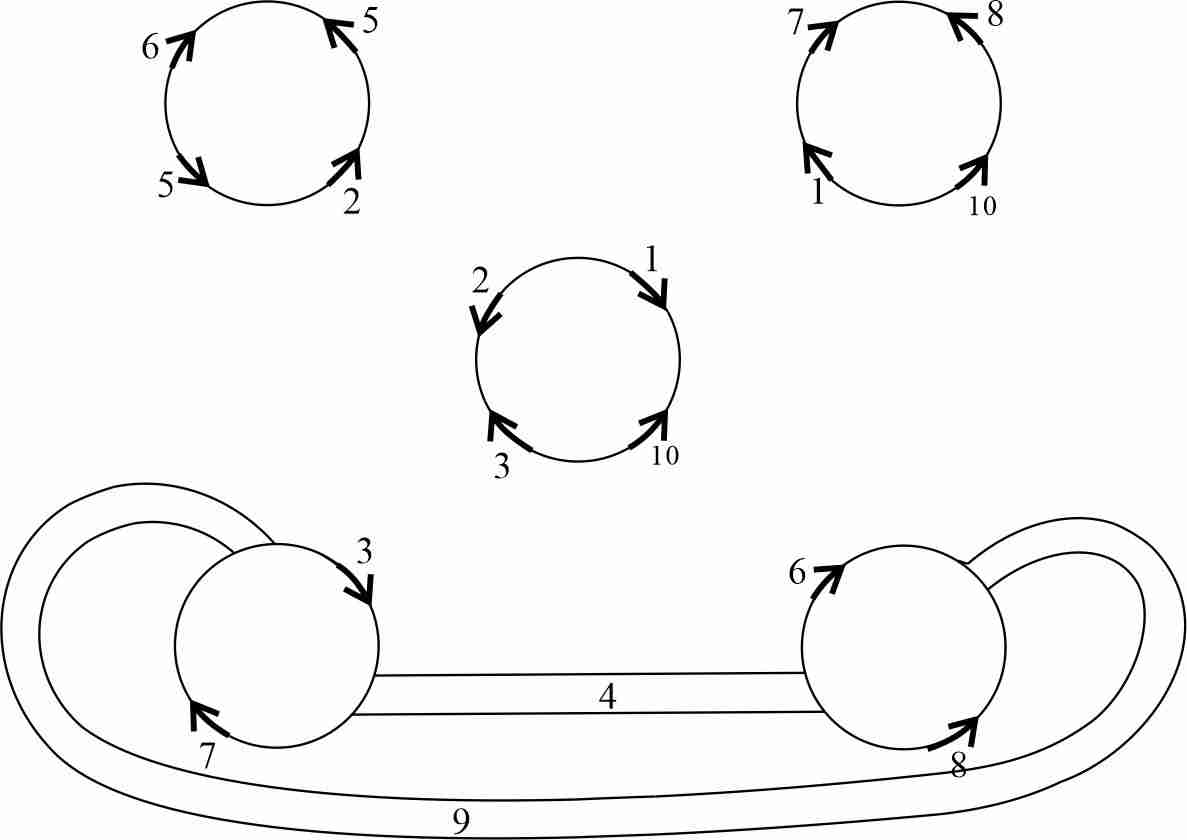}
\caption{Intermediate step for obtaining $G^{\tau(A_1)\delta(A_2)}$.}
\label{e44}
\end{figure}

\begin{figure}[htbp]
\centering
\includegraphics[width=2.5in]{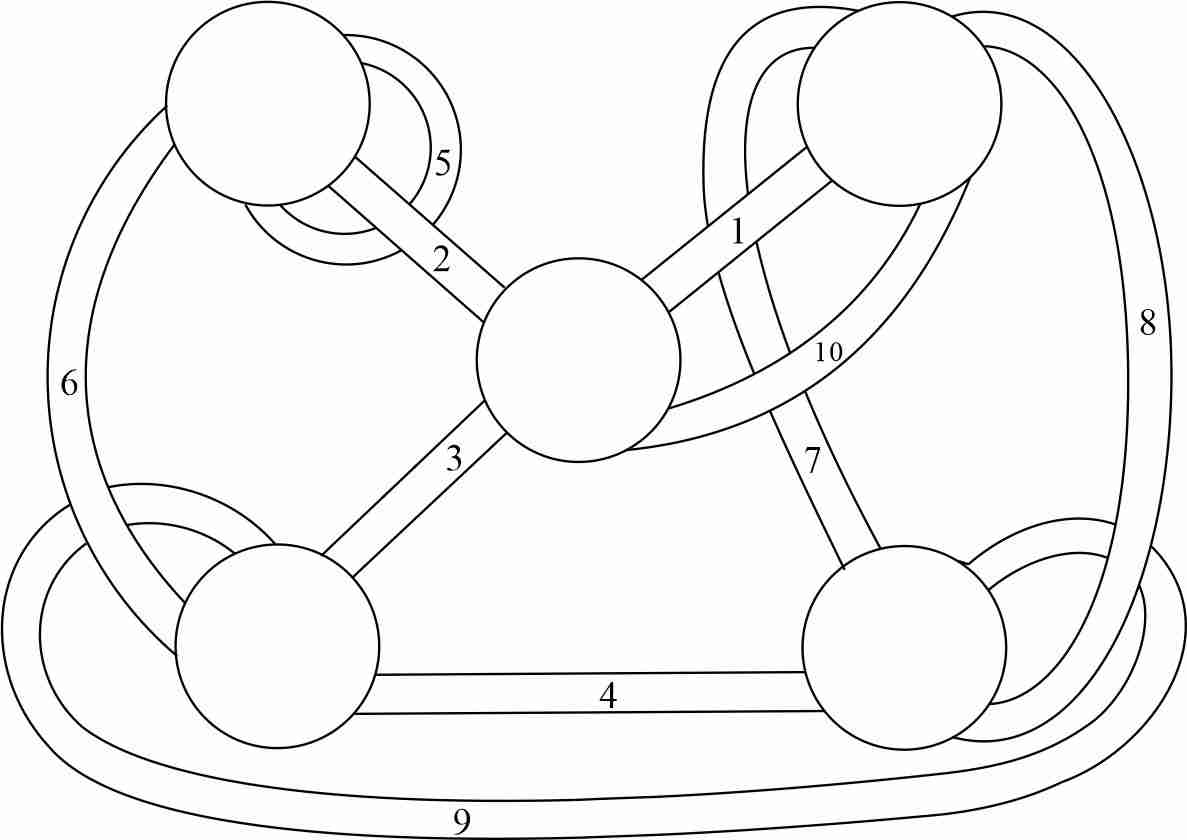}
\caption{$G^{\tau(A_1)\delta(A_2)}$.}
\label{e55}
\end{figure}

\begin{figure}[htbp]
\centering
\includegraphics[width=2in]{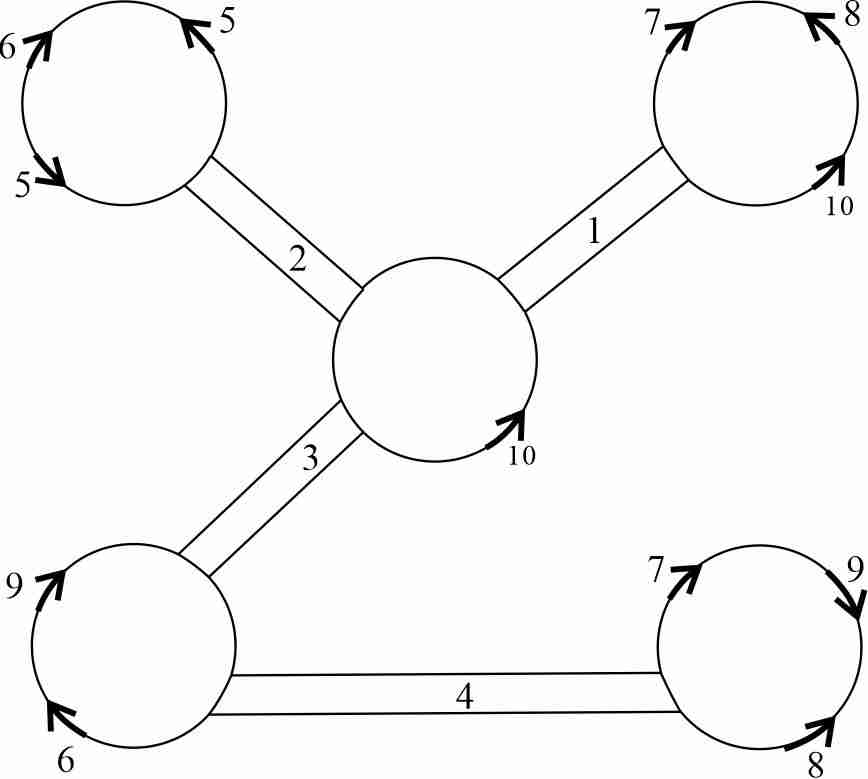}
\caption{Intermediate step for obtaining $\overline{G^{\tau(A_1)\delta(A_2)}/E(T)}$.}
\label{e66}
\end{figure}

\begin{figure}[htbp]
\centering
\includegraphics[width=2.3in]{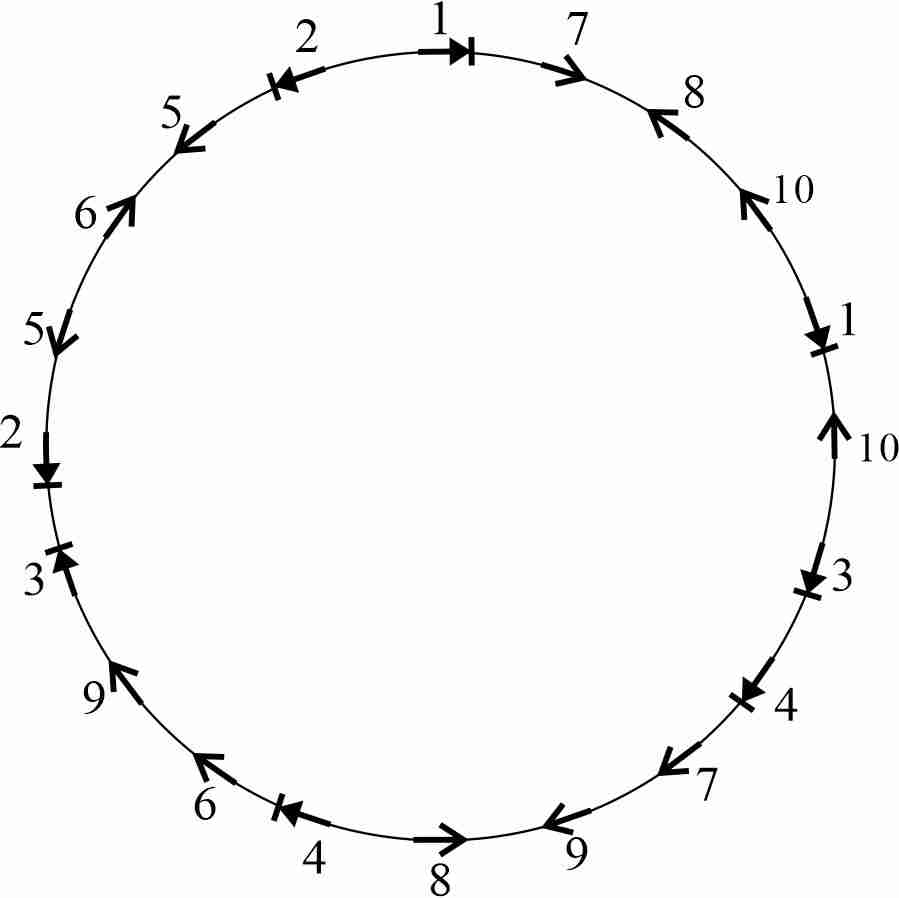}
\caption{$\overline{G^{\tau(A_1)\delta(A_2)}/E(T)}$.}
\label{e77}
\end{figure}

\begin{table}[t]
\centering
\begin{tabular}{|c|c|c|}\hline
  edge    &   \# of vertex line segments &  $t$  \\ \hline
  $e_5$&    $2$  &  $1$                          \\ \hline
  $e_6$&    $3$  &  $0$                          \\ \hline
  $e_7$&    $4$  &  $1$                          \\ \hline
  $e_8$&    $5$  &  $0$                          \\  \hline
  $e_9$&    $3$  &  $0$                          \\ \hline
  $e_{10}$& $1$  &  $0$                          \\  \hline
\end{tabular}
\caption{The index $t$.}\label{ta}
\end{table}

Finally we obtain $G^{\tau(A_1)\delta(A_2)\tau(A_3)}$ which is $4$-regular checkerboard colourable as shown in Figure \ref{e88}.

\begin{figure}[t]
\centering
\includegraphics[width=2.5in]{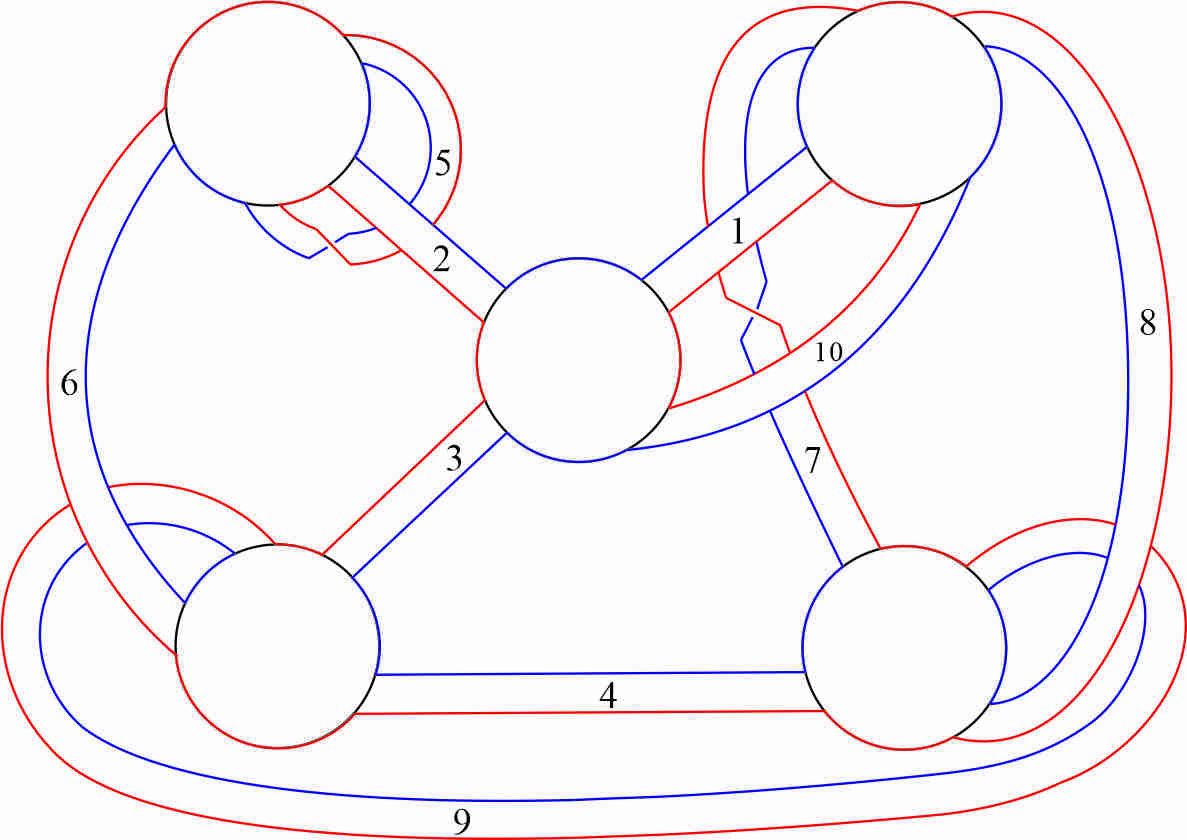}
\caption{$G^{\tau(A_1)\delta(A_2)\tau(A_3)}$.}
\label{e88}
\end{figure}
\end{example}

\bibliography{bibfile}

\end{document}